\documentclass[12pt, oneside]{article}   	
\usepackage[margin = 1in]{geometry}                		
\usepackage{graphicx}				
\usepackage{mathrsfs}
\usepackage{float}
\usepackage{pillowmath}
\usepackage{pgfplots}
\usepackage{authblk}

\usetikzlibrary{patterns}
\usepackage{thmtools}
\usepackage{thm-restate}

\newcount\gaussF
\edef\gaussR{0}
\edef\gaussA{0}

\makeatletter
\pgfmathdeclarefunction{gaussR}{0}{%
 \global\advance\gaussF by 1\relax
 \ifodd\gaussF
  \pgfmathrnd@%
  \ifdim\pgfmathresult pt=0.0pt\relax%
    \def\pgfmathresult{0.00001}%
  \fi
  \pgfmathln@{\pgfmathresult}%
  \pgfmathmultiply@{-2}{\pgfmathresult}%
  \pgfmathsqrt@{\pgfmathresult}%
  \global\let\gaussR=\pgfmathresult
  \pgfmathrnd@%
  \pgfmathmultiply@{360}{\pgfmathresult}%
  \global\let\gaussA=\pgfmathresult
  \pgfmathcos@{\pgfmathresult}%
  \pgfmathmultiply@{\pgfmathresult}{\gaussR}%
 \else
  \pgfmathsin@{\gaussA}%
  \pgfmathmultiply@{\gaussR}{\pgfmathresult}%
 \fi
}

\pgfmathdeclarefunction{invgauss}{2}{%
  \pgfmathln{#1}
  \pgfmathmultiply@{\pgfmathresult}{-2}%
  \pgfmathsqrt@{\pgfmathresult}%
  \let\@radius=\pgfmathresult%
  \pgfmathmultiply{6.28318531}{#2}
  \pgfmathdeg@{\pgfmathresult}%
  \pgfmathcos@{\pgfmathresult}%
  \pgfmathmultiply@{\pgfmathresult}{\@radius}%
}

\pgfmathdeclarefunction{randnormal}{0}{%
  \pgfmathrnd@
  \ifdim\pgfmathresult pt=0.0pt\relax%
    \def\pgfmathresult{0.00001}%
  \fi%
  \let\@tmp=\pgfmathresult%
  \pgfmathrnd@%
  \ifdim\pgfmathresult pt=0.0pt\relax%
    \def\pgfmathresult{0.00001}%
  \fi
  \pgfmathinvgauss@{\pgfmathresult}{\@tmp}%
}

\allowdisplaybreaks

\usepackage{hyperref}
\hypersetup{
    colorlinks=true,
    linkcolor=blue,
    filecolor=magenta,      
    urlcolor=cyan,
    citecolor=blue
}
\usetikzlibrary{positioning, decorations.pathreplacing}

\title{Laws of Large Numbers for Information Resolution}

\begin{document}

\author{Daniel Raban\footnote{Email: danielraban@berkeley.edu}}
\affil{Department of Statistics, University of California, Berkeley}


\maketitle



\abstract{Laws of large numbers establish asymptotic guarantees for recovering features of a probability distribution using independent samples. We introduce a framework for proving analogous results for recovery of the $\sigma$-field of a probability space, interpreted as information resolution—the granularity of measurable events given by comparison to our samples. Our main results show that, under iid sampling, the Borel $\sigma$-field in $\R^d$ and in more general metric spaces can be recovered in the strongest possible mode of convergence. We also derive finite-sample $L^1$ bounds for uniform convergence of $\sigma$-fields on $[0,1]^d$.

We illustrate the use of our framework with two applications: constructing randomized solutions to the Skorokhod embedding problem, and analyzing the loss of variants of random forests for regression.}

\section{Introduction}

Laws of large numbers generally assert that, in the context of iid sampling, we can asymptotically recover aspects of our probability space. For example, the strong and weak laws of large numbers assert that we can recover the mean of a measure $\mu$, and, perhaps more ambitiously, the Glivenko--Cantelli theorem guarantees recovery of the entire measure via its cumulative distribution function.

Inconspicuously absent from these theorems is the following consideration: Given a probability space $(S,\mc B,\mu)$ can recover the measure $\mu$ we were sampling from, but what about the $\sigma$-field $\mc B$? Does the information of our samples $X_i$ allow us to measure the same resolution of events as the unknown process associated to the samples?

The goal of this paper is to introduce a notion of laws of large numbers regarding recovery of the \emph{information resolution}, as represented by the notion of $\sigma$-fields, associated to the target measure $\mu$ generating our iid samples. Just as one approximates the underlying mean by a sample mean or the underlying CDF by an empirical CDF, we will approximate the underlying $\sigma$-field by \emph{empirical $\sigma$-fields}, representing the granularity of the events we can measure by comparison to our samples.

We will prove examples of these laws of large numbers in settings such as sampling in $\R^d$ and in more general metric spaces. We will also present two applications of our theory. The first gives a simple method for randomly generating solutions to the Skorokhod embedding problem by constructing stopping times for Brownian motion through sequences of hitting barriers, interpreted as increasingly resolving partitions of $\R$. The second applies our theory to random forests, analyzing how regression tree loss depends on tree depth by viewing feature space splits as progressively finer resolutions.

Here is a basic example illustrating the notion of information resolution.

\begin{ex}
Suppose we sample $X_1,X_2,X_3 \overset{\on{iid}}{\sim} \mu$ and get the values $X_1 = 5$, $X_2 = -4$, and $X_3 = 1$. What is the empirical resolution afforded by the knowledge of our three sample values? One choice is as follows: If we were to continue sampling $Z_1,Z_2,\dots \overset{\on{iid}}{\sim} \mu$, we would be able to compare the values of the $Z_i$ with our original sample values $X_1, X_2, X_3$. We would be able to determine events such as $\{ X_2 < Z_i \leq X_3\}$.
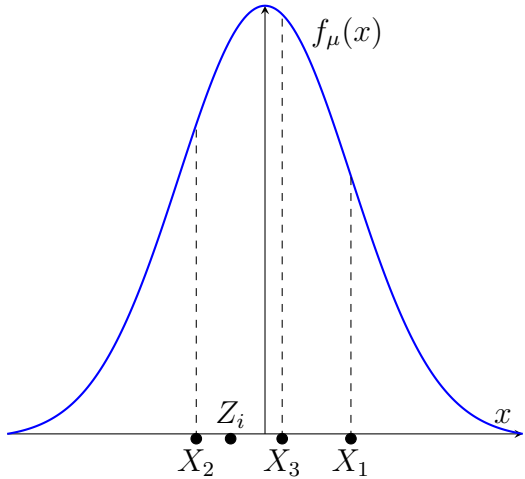
\begin{figure}[h]
\begin{center}
\begin{tikzpicture}
\begin{axis}[    domain=-15:15,    samples=100,    axis lines=middle,    xlabel=$\: x$,    ylabel=$\: \: \: \:\: f_\mu(x)$,    xtick=\empty,    ytick=\empty,    clip=false,    ]
\addplot [smooth, color=blue, thick] {exp(-x^2/50)/(5*sqrt(2*pi))};
\draw[dashed] (axis cs: 5,0) -- (axis cs: 5,{exp(-25/50)/(5*sqrt(2*pi))}) node[above] {};
\draw[dashed] (axis cs: -4,0) -- (axis cs: -4,{exp(-16/50)/(5*sqrt(2*pi))}) node[above] {};
\draw[dashed] (axis cs: 1,0) -- (axis cs: 1,{exp(-1/50)/(5*sqrt(2*pi))}) node[above] {};
\filldraw (axis cs: 5,0) circle (2pt) node[below] {$X_1$};
\filldraw (axis cs: -4,0) circle (2pt) node[below] {$X_2$};
\filldraw (axis cs: 1,0) circle (2pt) node[below] {$X_3$};
\filldraw (axis cs: -2,0) circle (2pt) node[above] {$Z_i$};
\end{axis}
\end{tikzpicture}
\end{center}
\caption{Comparing a new sample $Z_i$ to the previous samples $X_1,X_2,X_3$.}
\end{figure}

From this perspective, the $\sigma$-field representing the resolution given by our first three samples is the $\sigma$-field generated by the partition
$$\mc F_3 \coloneqq \sigma((-\infty,-4], (-4,1], (1,5], (5,\infty) ).$$
Alternatively, we can express this $\sigma$-field more directly in terms of our samples using the sets $(-\infty,X_i]$:
$$\mc F_3 = \sigma((-\infty,5], (-\infty,-4], (-\infty,1]).$$
Defining empirical $\sigma$-fields in this way, i.e.\ $\mc F_n \coloneqq \sigma((-\infty,X_1],\dots,(-\infty,X_n])$, we can measure more events as we obtain more samples, increasing the granularity of our information resolution. And as we let $n \to \infty$, we might hope that we can measure any event.
\end{ex}

Care must be taken, however, when defining a notion of empirical information resolution, as not every sequence of $\sigma$-fields will recover the maximal $\sigma$-field of the probability space. Here is a naive example illustrating this point.

\begin{ex}
When sampling $X_1,X_2,X_3 \overset{\on{iid}}{\sim} \mu = \on{Unif}[0,1]$, we could define $\mc G_n \coloneqq \sigma(\{ X_1\},\dots,\{ X_n\})$. This, at best, generates a sub-$\sigma$-field of $\mc C$, the $\sigma$-field of countable and co-countable subsets of $[0,1]$. Moreover, all sets in $\mc C$ have Lebesgue measure 0 or 1, so from the perspective of Lebesgue measure on $[0,1]$, we have not gained any resolution at all. The sequence $\mc G_n$ of empirical $\sigma$-fields would only be sufficient for recovering the resolution of our space if the probability measure $\mu$ were discrete.
\end{ex}

In general, the setup for $\sigma$-field recovery is as follows: draw iid samples $X_1, X_2, \dots \overset{\on{iid}}{\sim} \mu$, taking values in a space $S$. To each $x \in S$, we associate a set $A_x$ that reflects the resolution or information revealed by observing $x$. These sets encode our assumption about the underlying structure, with the goal of recovering the maximal $\sigma$-field $\mc F \coloneqq \sigma(\{A_x : x \in S\})$. We define the \emph{empirical resolution} $\sigma$-fields $\mc F_n \coloneqq \sigma(A_{X_1}, \dots, A_{X_n})$, based on the first $n$ samples. The central question is whether $\mc F_n$ converges to $\mc F$ as $n \to \infty$, under an appropriate notion of convergence for $\sigma$-fields.

Convergence of $\sigma$-fields has been well-studied (see e.g.\ \cite{boylan1971equiconvergence,neveu1972note,kudo1974note,rogge1974uniform,van1993hausdorff,vidmar2018couple}), and there are a number of non-equivalent modes of convergence. Most of these modes of convergence involve comparing the $\sigma$-fields using a fixed measure $\mu$, which we will usually assume to be the shared marginal distribution of our iid samples. We list some modes of convergence here; for a more in-depth study of how these notions relate to each other, see \cite{vidmar2018couple}, for example.

\begin{itemize}

\item Monotone convergence: $\mc F_n \to \mc F$ in the monotone sense (written $\mc F_n \uparrow \mc F$) means that $\bigvee_{n=1}^\infty \mc F_n = \mc F$. Here, $\bigvee_{n=1}^\infty \mc F_n$ is the \emph{join} of these $\sigma$-fields with respect to inclusion; that is, it is the smallest $\sigma$-field containing $\mc F_n$ for each $n$.

\item Hausdorff convergence: Given a fixed probability measure $\mu$, $\mc F_n \to \mc F$ in the Hausdorff sense means that
$$d_\mu(\mc F_n, \mc F) \coloneqq \sup_{\| f\|_{L^\infty(\mu)} \leq 1} \| \E[ f \mid \mc F_n] - \E[f \mid \mc F] \|_{L^1(\mu)} \to 0.$$
This is equivalent \cite{rogge1974uniform} to
$$d'_\mu(\mc F_n, \mc F) \coloneqq \max \brac{\sup_{A \in \mc F_n} \inf_{B \in \mc F} \mu(A \triangle B), \sup_{B \in \mc F} \inf_{A \in \mc F_n} \mu(A \triangle B)} \to 0,$$
which is convergence of the sets $\mc F_n$ to $\mc F$ in the Hausdorff topology induced by viewing these $\sigma$-fields as closed subsets of $L^1$ (via indicator functions of sets).

\item Set theoretic convergence: This means  $\limsup_{n \to \infty} \mc F_n = \liminf_{n \to \infty} \mc F_n = \mc F$, where
$$\limsup_{n \to \infty} \mc F_n \coloneqq \bigcap_{n=1}^\infty \bigvee_{k=n}^\infty \mc F_n, \qquad \liminf_{n \to \infty} \mc F_n \coloneqq \bigvee_{n=1}^\infty \bigcap_{k=n}^\infty \mc F_n.$$

\item Strong convergence: This means $\E[ \mbbm 1_A \mid \mc F_n] \to \E[\mbbm 1_A \mid \mc F]$ in probability for all measurable $A$.

\end{itemize}
In general, monotone and Hausdorff convergence, which are not equivalent, are the strongest. Here is a diagram expressing the strength of various modes of convergence, including some not mentioned above; for a more complete picture, see \cite{vidmar2018couple}.
$$\begin{tikzcd}
& \text{Monotone} \arrow[d] \arrow[dl] \\
\text{Almost sure} \arrow[dr] & \text{Set theoretic} \arrow[d] & \text{Hausdorff} \arrow[dl] \\
& \text{Strong} \arrow[d] \\
& \text{Weak}
\end{tikzcd}$$

Hausdorff convergence, which is given by a pseudometric, may at first seem the most natural to use for an analogue of the Glivenko--Cantelli theorem, due to its uniform nature. However, we will see in Section~\ref{uniform-section} that Hausdorff convergence fails for even simple examples in $\R$. Monotone convergence, which appears regularly in probability theory (for example, in the context of martingale convergence), is another natural choice and will suffice in cases where Hausdorff convergence is not possible.

\paragraph*{Outline.} In what follows, we will prove laws of large numbers for two modes of convergence of $\sigma$-fields.

\begin{itemize}[label = -]

\item In Section~\ref{monotone-section}, we prove theorems for monotone convergence of $\sigma$-fields in $\R^d$ and in more general metric spaces. This gives the strongest convergence possible, as monotone convergence implies all studied modes of convergence for $\sigma$-fields which do not imply Hausdorff convergence.

\item In Section~\ref{uniform-section}, we prove a weakened version of Hausdorff convergence (and give quantitative rates) by restricting the class of test functions to Lipschitz functions, rather than all of $L^\infty(\mu)$; in other words, we give bounds on
$$\sup_{\| f \|_{\on{Lip}} \leq 1} \| \E[ f \mid \mc F_n] - \E[ f \mid \mc F] \|_{L^1(\mu)}.$$

\item In Section~\ref{app-section}, we apply our theorems to construct randomized solutions to the Skorokhod embedding problem and to analyze the loss of randomized regression trees. These applications use our theorems from Sections~\ref{monotone-section} and \ref{uniform-section}, respectively.

\end{itemize}

It is important to note that there are two layers of randomness at play: We want to study a probability space $(S,\mc F,\mu)$, but we are generating the samples $X_1,X_2,\dots \overset{\on{iid}}{\sim} \mu$ via some background probability space $(\Omega, \mc G,\P)$. Just as classical laws of large numbers concern $\P$-a.s.\ convergence of numbers or random measures, our theorems will concern $\P$-a.s.\ and $L^1(\P)$ convergence of random $\sigma$-fields.

\section{Monotone convergence of resolution} \label{monotone-section}

When studying the convergence of $\sigma$-fields, we want to compare $\sigma$-fields by measuring the distance between sets with respect to a \emph{fixed} measure $\mu$. The measure $\mu$ can't meaningfully distinguish between two sets $A,B$ with $\mu(A \triangle B) = 0$, so we will need to be precise with our statements. However, the following definition and subsequent proposition tell us that this technicality poses no obstruction to our understanding.

\begin{defn}
Let $(S,\mc F,\mu)$ be a measure space, and let $\mc A,\mc B \subseteq \mc F$. We say that $\mc A$ and $\mc B$ \emph{differ only by $\mu$-null sets} if
\begin{enumerate}[label = (\roman*)]
\item $\forall A \in \mc A, \exists B \in \mc B \text{ s.t.\ } \mu(A \triangle B) = 0$,
\item $\forall B \in \mc B, \exists A \in \mc A \text{ s.t.\ } \mu(A \triangle B) = 0$.
\end{enumerate}
\end{defn}

We will make judicious use of the generating construction for $\sigma$-fields: $\sigma(\mc A)$ denotes the smallest $\sigma$-field containing all the sets in $\mc A$, and we say that $\mc A$ \emph{generates} $\sigma(\mc A)$. Before proving any results, we must first make sure that altering generating sets by null sets does not cause any issues when generating $\sigma$-fields.

\begin{prop} \label{null-generating-fine}
Let $(S,\mc F,\mu)$ be a measure space, and let $\mc A,\mc B \subseteq \mc F$ differ only by $\mu$-null sets. Then $\sigma(\mc A)$ and $\sigma(\mc B)$ differ only by $\mu$-null sets.
\end{prop}

\begin{proof}
Let $\mc F \coloneqq \{ A \in \sigma(\mc A) : \exists B \in \sigma(\mc B) \text{ s.t.\ } \mu(A \triangle B) = 0\}$ be the members of $\sigma(\mc A)$ which are represented in $\sigma(\mc B)$ up to null sets. Then $\mc F$ is a $\sigma$-field:
\begin{enumerate}[label = (\roman*)]

\item Empty set: $\varnothing \in \mc F$ because $\varnothing \in \sigma(\mc A)$ and $\sigma(\mc B)$.

\item Complements: If $A \in \mc F$, then letting $B$ be such that $\mu(A \triangle B) = 0$, we get $\mu(A^c \triangle B^c) = 0$. As $B^c \in \sigma(\mc B)$, we get $A^c \in \mc F$.

\item Countable unions: If $A_1,A_2,\dots \in \mc F$, then let $B_1,B_2,\dots \in \sigma(\mc B)$ be such that $\mu(A_i \triangle B_i) = 0$ for $i  \geq 1$. Then
$$\mu \paren{ \paren{\bigcup_{i =1}^\infty A_i } \triangle \paren{\bigcup_{i=1}^\infty B_i}} \leq \mu \paren{ \bigcup_{i =1}^\infty A_i \triangle B_i} \leq \sum_{i=1}^\infty \mu(A_i \triangle B_i) = 0,$$
so $\bigcup_{i =1}^\infty A_i  \in \sigma(\mc A)$.

\end{enumerate}
$\mc F$ is a $\sigma$-field that contains $A$, so $\mc F \supseteq \sigma(\mc A)$. Hence, $\mc F = \sigma(\mc A)$. The same argument shows that all members of $\sigma(\mc B)$ are represented in $\sigma(\mc A)$ up to null sets.
\end{proof}

\subsection{Monotone convergence of resolution in $\R^d$}

In this section, we prove a basic law of large numbers for recovering the Borel $\sigma$-field in $\R^d$, using the left-infinite intervals/boxes which show up in the Glivenko--Cantelli theorem.

\begin{thm} \label{classical R^n G-C}
Let $(\R^d,\mc B,\mu)$ be a probability space equipped with the Borel $\sigma$-field, and let $X_1,X_2,\dots \overset{\on{iid}}{\sim} \mu$. For $x = (x_1,\dots,x_d) \in \R^d$, let $A_x \coloneqq (-\infty,x_1] \times \cdots \times (-\infty,x_d]$, and define the empirical $\sigma$-fields $\mc F_n \coloneqq \sigma(A_{X_1},\dots,A_{X_n})$. Then $\mc F_n \uparrow \mc B$ a.s.; that is, $\bigvee_{n=1}^\infty \mc F_n$ and $\mc B$ differ only by $\mu$-null sets.
\end{thm}

This choice of $A_x$ is, of course, not the only choice that works. The proof works essentially the same with finite-sized boxes, balls, etc. To recover a different $\sigma$-field, one would use a different choice of $A_x$ sets; the choice of $A_x = (-\infty,x_1] \times \cdots \times (-\infty,x_d]$ necessarily implies that we are attempting to recover a sub-$\sigma$-field of the Borel $\sigma$-field because $\sigma(\{ A_x : x \in \R^d\}) = \mc B$.

\begin{lem} \label{inf-achieved}
Let $\mc G \subseteq \mc F$ be $\sigma$-fields, let $\mu$ be a probability measure defined on $\mc F$, and let $B \in \mc F$. Then the infimum
$$\inf_{A \in \mc G} \mu(A \triangle B)$$
is achieved.
\end{lem}

\begin{proof}[Proof of Lemma~\ref{inf-achieved}]
To construct the minimizing set, we round $\E[\mbbm 1_B \mid \mc G]$ to get the ``closest indicator.'' Let $A_* = \{ x \in S : \E[\mbbm 1_B \mid \mc G] \geq 1/2 \}$ for some version of $ \E[\mbbm 1_B \mid \mc G]$ as a member of $L^2(\mu)$. We can directly show that $\mu(A_* \triangle B) \leq \mu(A \triangle B)$ for any $A \in \mc G$:
\begin{align*}
\mu(A_* \triangle B) &= \| \mbbm 1_{A_*} - \mbbm 1_B \|_{L^1(\mu)} \\
&= \| \mbbm 1_{A_*} - \mbbm 1_B \|_{L^2(\mu)}^2
\shortintertext{By the Pythagorean theorem,}
&= \norm{ \mbbm 1_{A_*} - \E\squa{\mbbm 1_B \mid \mc G}  }_{L^2(\mu)}^2 + \norm{ \E \squa{\mbbm 1_B \mid \mc G} - \mbbm 1_B }_{L^2(\mu)}^2
\shortintertext{By definition, for $\mu$-a.e.\ $x \in S$, $\mbbm 1_{A_*}(x)$ is closer to $\E[\mbbm 1_B \mid \mc G](x)$ than any other $\mc G$-measurable indicator is.}
&\leq \norm{ \mbbm 1_{A} - \E \squa{\mbbm 1_B \mid \mc G}  }_{L^2(\mu)}^2 + \norm{ \E \squa{\mbbm 1_B \mid \mc G} - \mbbm 1_B }_{L^2(\mu)}^2 \\
&= \| \mbbm 1_{A} - \mbbm 1_B \|_{L^2(\mu)}^2 \\
&= \| \mbbm 1_{A} - \mbbm 1_B \|_{L^1(\mu)} \\
&= \mu(A \triangle B).\qedhere
\end{align*}
\end{proof}

Taking the case where this infimum is zero gives the following topological interpretation of the above lemma.

\begin{cor}[$L^p(\mu)$ Closure of $\sigma$-fields] \label{sigma-fields-are-closed}
Let $\mc G \subseteq \mc F$ be $\sigma$-fields, let $\mu$ be a probability measure defined on $\mc F$, and let $B \in \mc F$. If there exists a sequence $B_n \in \mc G$ such that $\mu(B_n \triangle B) \to 0$ as $n \to \infty$, then $B \in \mc G$. In other words, $\{ \mbbm 1_A : A \in \mc G \}$ is a closed subset of $L^p(\mu)$ for all $1 \leq p < \infty$.
\end{cor}

\begin{proof}[Proof of Theorem~\ref{classical R^n G-C}]
The idea is to reduce the problem to showing that our empirical $\sigma$-fields can approximate any box. Then we use the Glivenko--Cantelli theorem to approximate any box from the inside; see Figure~\ref{Rn-proof-figure} for a picture.
\begin{itemize}

\item[] Step 1. (Reduce to recovering generating boxes): Since $\mc B$ is generated by the countable collection $\{A_q : q \in \Q^d\}$, Proposition~\ref{null-generating-fine} reduces the problem to showing that for each $q \in \Q^d$, with probability 1, there exists $A \in \bigvee_{n=1}^\infty \mc F_n$ such that $\mu(A \triangle A_q) = 0$.

\item[] Step 2. (Reduce to approximating non-null boxes): By Corollary~\ref{sigma-fields-are-closed}, it suffices to show that $\inf_{A \in \bigvee_{n=1}^\infty \mc F_n} \mu(A \triangle A_q) = 0$ almost surely. Fix $q \in \Q^d$ and $\varepsilon > 0$. We will exhibit a set $A \in \bigvee_{n=1}^\infty \mc F_n$ such that $\mu(A \triangle A_q) < \varepsilon$. Moreover, we may assume that $\mu(A_q) \neq 0$; otherwise, we can just pick $A = \varnothing$ and be done.

\item[] Step 3. (Approximate boxes from inside): Consider the empirical measure $\mu_N \coloneqq \frac{1}{n} \sum_{n=1}^N \delta_{X_n}$. From the Glivenko--Cantelli theorem, we can choose $N$ such that $\sup_{x \in \R^d} |\mu_N(A_x) - \mu(A_x)| < \varepsilon/2$. For non-null $A_q$, $\P(X_n \notin A_q \: \forall n) = 0$, so we may assume, increasing $N$ if necessary, that $A_q$ contains $X_n$ for some $n \leq N$.

Using this value of $N$, define $r = (r_1,\dots,r_d) \in \R^d$ by $r_i \coloneqq \max \{ (X_j)_i : X_j \in A_q,1 \leq j \leq N \}$. Then $\mu_N(A_r) = \mu_N(A_q)$, and $A_r \subseteq A_q$, so we can write
\begin{align*}
\mu(A_r \triangle A_q) &= \mu(A_q) - \mu(A_r) \\
&= \underbrace{\mu(A_q) - \mu_N(A_q)}_{< \varepsilon/2} + \underbrace{\mu_N(A_q) - \mu_N(A_r)}_{=0} + \underbrace{\mu_N(A_r) - \mu(A_r)}_{< \varepsilon/2} \\
&< \varepsilon.\qedhere
\end{align*}
\end{itemize}
\end{proof}

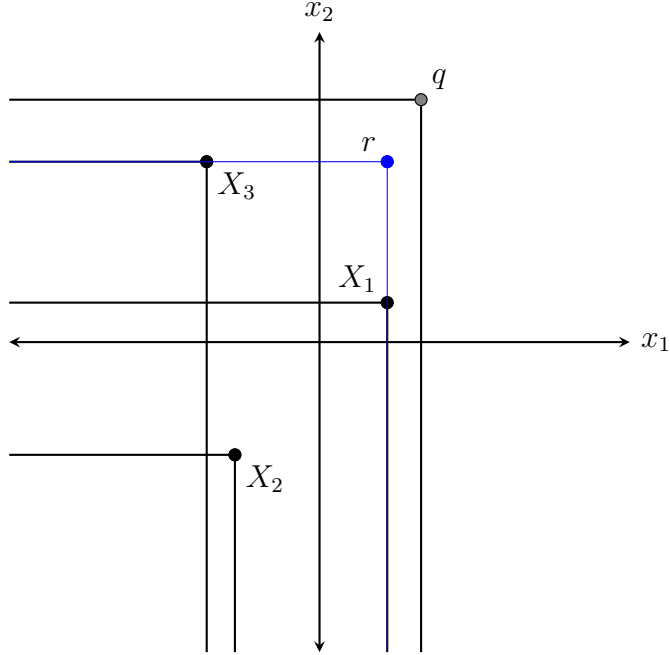
\begin{figure}[h]
\begin{center}
\begin{tikzpicture}[scale = 0.75]

  \draw[<->, thick, >=stealth] (-5.5,0) -- (5.5,0) node[right] {$x_1$};
  \draw[<->, thick, >=stealth] (0,-5.5) -- (0,5.5) node[above] {$x_2$};
  
  \draw[thick] (1.2,0.7) -- (-5.5,0.7);
  \draw[thick] (1.2,0.7) -- (1.2,-5.5);
  
  \filldraw[black](1.2,0.7) circle (3pt);
  \draw[black](1.2,0.7) circle (3pt);
  \node[above left] at (1.2,0.7) {$X_1$};

  \draw[thick] (1.8,4.3) -- (-5.5,4.3);
  \draw[thick] (1.8,4.3) -- (1.8,-5.5);
  
  \filldraw[gray](1.8,4.3) circle (3pt);
  \draw[black](1.8,4.3) circle (3pt);
  \node[above right] at (1.8,4.3) {$q$};

  \draw[thick] (-2,3.2) -- (-5.5,3.2);
  \draw[thick] (-2,3.2) -- (-2,-5.5);
  
  \filldraw[black](-2,3.2) circle (3pt);
  \draw[black](-2,3.2) circle (3pt);
  \node[below right] at (-2,3.2) {$X_3$};

   \draw[thick] (-1.5,-2) -- (-5.5,-2);
  \draw[thick] (-1.5,-2) -- (-1.5,-5.5);
  
  \filldraw[black](-1.5,-2) circle (3pt);
  \draw[black](-1.5,-2) circle (3pt);
  \node[below right] at (-1.5,-2) {$X_2$};
  
   \draw[color=blue] (1.2,3.2) -- (-5.5,3.2);
  \draw[color=blue] (1.2,3.2) -- (1.2,-5.5);
  
  \filldraw[blue](1.2,3.2) circle (3pt);
  \draw[blue](1.2,3.2) circle (3pt);
  \node[above left] at (1.2,3.2) {$r$};

\end{tikzpicture}
\end{center}
\caption{Approximating a box $A_q$ from inside in the proof of Theorem~\ref{classical R^n G-C}.}
\label{Rn-proof-figure}
\end{figure}

\subsection{Monotone convergence of resolution in metric spaces}

Before extending our viewpoint to the more general setting of metric spaces, we must first review some technical notions regarding regularity of measures. The following definition is from \cite{rigot2021differentiation}.

\begin{defn}
Let $\mu$ be a measure on a metric space $(S,\rho)$. We say that $\mu$ is of \emph{Vitali type with respect to $\rho$} if for every $A \subseteq S$ and every family $\mc C$ of balls in $(S,\rho)$ such that $\inf \{ r > 0 : B(x,r) \in \mc C \} = 0$ for all $x \in A$, there exists a countable subfamily $\mc D \subseteq \mc C$ of disjoint balls for which
$$\mu \paren{A \setminus \bigcup_{B \in \mc D} B} = 0.$$
\end{defn}

\cite{rigot2021differentiation} provides a number of examples with this property. Here are a few classes of examples.

\begin{ex}
Any Radon measure on $\R^d$ is of Vitali type with respect to the Euclidean metric.
\end{ex}

\begin{ex}
Every probability measure $\mu$ on $(S,\rho)$ which is \emph{doubling} is of Vitali type with respect to $\rho$. Here, $\mu$ is said to be doubling if there exists a constant $C \geq 1$ such that
$$\mu(B_{2r}(x)) \leq C \mu(B_r(x)) \qquad \forall x \in S, r > 0.$$
\end{ex}

The reason we care about the Vitali type property is that it describes the regularity of the density of a set $A$ with respect to the measure $\mu$. In particular, it tells us that the measure $\mu$ enjoys an analogue of the Lebesgue differentiation theorem.

\begin{lem}[\cite{rigot2021differentiation}] \label{leb-diff}
Let $\mu$ be a measure which is of Vitali type with respect to a metric space $(S,\rho)$. Then for every measurable set $A$,
$$\lim_{r \downarrow 0} \frac{\mu(A \cap B_r(x))}{\mu(B_r(x))} = \mbbm 1_A(x) \qquad \text{ for $\mu$-a.e.\ $x \in S$}.$$
\end{lem}

When generalizing the ideas of the previous section to metric spaces, we lose the helpful ordering of $\R$. The natural candidate for a set $A_x$ in a general metric space is a ball $B_r(x)$ of radius $r$, centered at $x$. However, the following simple example shows that balls of a fixed radius may not always suffice.

\begin{ex}
Consider the metric space $[0,1]$  with the Euclidean metric and the measure $\mu(\{1/k\}) = 2^{-k}$ for $k = 1,2,\dots$. If we set $A_x = B_r(x)$ for any $r>0$, then we there are some points we cannot distinguish.
\end{ex}

However, we can still recover information resolution by sampling balls of varying radii. To make sure we can obtain a ball of any arbitrarily small radius, we introduce auxiliary randomness, which can be interpreted as a degree of noise determining the resolution given by the sample point $X_n$.

\begin{thm} \label{metric G-C}
Let $(S,\rho,\mc B,\mu)$ be a separable metric space equipped with the Borel $\sigma$-field and a probability measure $\mu$ which is of Vitali type with respect to $\rho$. Let $X_1,X_2,\dots \overset{\on{iid}}{\sim} \mu$ and $R_1,R_2,\dots \overset{\on{iid}}{\sim} \nu$ be independent, where $\nu$ is a distribution on $\R_{\geq 0}$ with $\nu((0,\varepsilon)) > 0$ for every $\varepsilon > 0$. For $x \in S$ and $r > 0$, let $A_{x,r} \coloneqq B_r(x) = \{ z \in S : \rho(z,x) < r\}$, and define the empirical $\sigma$-fields $\mc F_n \coloneqq \sigma(A_{X_1,R_1},\dots,A_{X_n,R_n})$. Then $\mc F_n \uparrow \mc B$ a.s.; that is, $\bigvee_{n=1}^\infty \mc F_n$ and $\mc B$ differ only by $\mu$-null sets.
\end{thm}

\begin{rem}
The metric structure is not entirely essential in Theorem~\ref{metric G-C}. We mainly restrict this theorem to metric spaces to express the regularity of $\mu$ via the notion of set density with respect to $\mu$. This proof technique would work for any choice of sampling sets $A_{x,r}$ with appropriate regularity for $\mu$ as the sets $A_{x,r}$ more closely approximate $x$, e.g., a countable neighborhood base for a second countable topological space when $\mu$ is purely atomic. In fact, the $\sigma$-field need not be the Borel $\sigma$-field in general!
\end{rem}

\begin{proof}
Let $C$ be a countable dense subset of $S$. Balls of rational radius centered at points in $C$ generate $\mc B$, so it suffices to show that if $c \in C$ and $r \in \Q_{>0}$, $\bigvee_{n=1}^\infty \mc F_n$ contains $B_r(c)$ a.s. As in the proof of Theorem~\ref{classical R^n G-C}, it suffices for us to show that $\inf_{B' \in \bigvee_{n=1}^\infty \mc F_n} \mu(B_r(c) \triangle B') = 0$ a.s. We will show that the complement event has probability 0.

Suppose that $\inf_{B' \in \bigvee_{n=1}^\infty \mc F_n} \mu(B_r(c) \triangle B') = \delta > 0$. Then, by Lemma~\ref{inf-achieved}, there exists some $B_* \in \bigvee_{n=1}^\infty \mc F_n$ with $\mu(B_r(c) \triangle B_*) = \delta$. Without loss of generality, we may assume that $\mu(B_r(c) \setminus B_*) > 0$; the argument for $B_* \setminus B_r(c)$ is analogous. Lemma~\ref{leb-diff} provides a set $U \subseteq B_r(c) \setminus B_*$ of positive measure which only contains points of positive density with respect to $B_r(c)$:
$$\lim_{t \downarrow 0} \frac{\mu((B_r(c) \setminus B_*) \cap B_t(x))}{\mu(B_t(x))} = 1 \qquad \forall x \in U.$$
Hence, for each $x \in U$, there exists some radius $r_x$ such that for $t \leq r_x$,
$$\frac{\mu((B_r(c) \setminus B_*) \cap B_t(x))}{\mu(B_t(x))} > 1/2.$$
Rearranging gives
$$\mu((B_r(c) \setminus B_*) \cap B_t(x)) > \mu(B_t(x) \setminus (B_r(c) \setminus B_*)).$$

On the other hand, disintegrating over $U$ gives
\begin{align*}
\P(X_n \in U, R_n \leq r_{X_n}) &= \int_U \underbrace{\P(R_n \leq r_{x})}_{>0} d\mu(x) \\
&> 0,
\end{align*}
so that $\P(\exists n \text{ s.t.\ } X_n \in U, R_n \leq r_{X_n}) = 1$. This event is inconsistent with the fact that $\mu(B_r(c) \triangle B_*) = \delta$ because it implies that we could take $B_{**} \coloneqq B_* \cup B_{R_{X_n}}(X_n)$ for some $n$ and get the improved approximation
\begin{align*}
\mu(B_r(c) \triangle B_{**}) &\leq \underbrace{\mu(B_r(c) \triangle B_*)}_{=\delta} \\
&\quad + \underbrace{\mu( B_{R_{X_n}}(X_n)) \setminus (B_r(c) \setminus B_*)) - \mu((B_r(c) \setminus B_*) \cap B_{R_{X_n}}(X_n))}_{< 0} \\
&< \delta,
\end{align*}
contradicting the optimality of $B_*$. So $\P(\inf_{B' \in \bigvee_{n=1}^\infty \mc F_n} \mu(B_r(c) \triangle B') > 0) = 0$, as claimed.
\end{proof}

\section{Uniform convergence of resolution} \label{uniform-section}

If $\mc F_n \uparrow \mc F$, the martingale convergence theorem gives $\E[f \mid \mc F_n] \to \E[f \mid \mc F]$ a.s.\ and in $L^1$ for all bounded $f$. Hausdorff convergence can be viewed as a uniform version of this convergence:
$$d_\mu(\mc F_n, \mc F) \coloneqq \sup_{\| f\|_{L^\infty(\mu)} \leq 1} \| \E[ f \mid \mc F_n] - \E[f \mid \mc F] \|_{L^1(\mu)} \to 0.$$
However, uniform convergence over the entire unit ball in $L^\infty(\mu)$ is too strong of a condition for our purposes, as the following example shows.

\begin{ex}
Consider the probability space $([0,1],\mc B,\text{Leb})$, where $\mc B$ is the Borel $\sigma$-field and $\lambda$ is Lebesgue measure. Given any realization $\mc F_n \coloneqq \sigma([0,x_1],\dots,[0,x_n])$ of an empirical $\sigma$ field, we adversarially construct a function $f_n$ as follows: Let $0 < x_{(1)} < \cdots < x_{(n)} < 1$ list the sample points in increasing order, and take the convention that $x_{(0)} = 0$ and $x_{(n+1)} = 1$. Define
$$f_n(x) = \begin{cases}
1 &\text{if } x_{(k)} \leq x < \frac{x_{(k)} + x_{(k+1)}}{2} \text{ for some $0 \leq k \leq n$} \\
-1 &\text{if } \frac{x_{(k)} + x_{(k+1)}}{2} \leq x < x_{(k+1)} \text{ for some $0 \leq k \leq n$}.
\end{cases}$$
See Figure~\ref{adversarial-f} for an illustration.

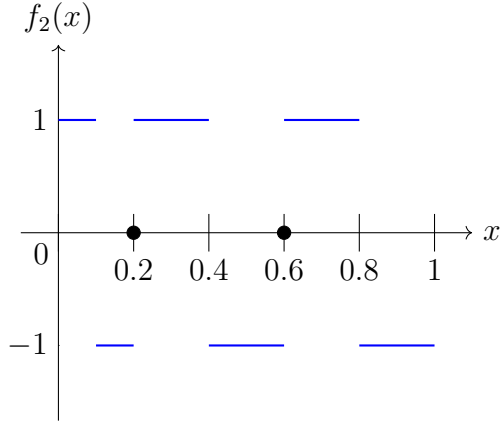
\begin{figure}[h]
\begin{center}
\begin{tikzpicture}[scale=5]
\draw[->] (-0.1,0) -- (1.1,0) node[right] {$x$};
\draw[->] (0,-0.5) -- (0,0.5) node[above] {$f_2(x)$};
\foreach \x in {0.2, 0.4,0.6,0.8,1} {
  \draw (\x,-0.05) -- (\x,0.05) node[below=13pt] {\x};
}
\draw (0,-0.3) -- (0,-0.3) node[left=0pt] {$-1$};
\draw (0,0.3) -- (0,0.3) node[left=0pt] {$1$};
\draw (0,-0.05) -- (0,0.05) node[below=7pt] {0 \: \:};
\draw[blue, thick] (0,0.3) -- (0.1,0.3);
\draw[blue, thick] (0.1,-0.3) -- (0.2,-0.3);
\draw[blue, thick] (0.2,0.3) -- (0.4,0.3);
\draw[blue, thick] (0.4,-0.3) -- (0.6,-0.3);
\draw[blue, thick] (0.6,0.3) -- (0.8,0.3);
\draw[blue, thick] (0.8,-0.3) -- (1,- 0.3);
\filldraw ( 0.2,0) circle (0.5pt) node[below] {};
\filldraw ( 0.6,0) circle (0.5pt) node[below] {};
\end{tikzpicture}
\end{center}
\caption{An adversarially chosen function which maximizes the Hausdorff distance.}
\label{adversarial-f}
\end{figure}
Then on each $A \in \mc F_n$, $\E[f_n \mid A] = 0$, so $\E[f_n \mid \mc F_n] = 0$ $\lambda$-a.s. Thus,
$$d_{\lambda}(\mc F_n, \mc B) \geq \| \E[f_n \mid \mc F_n] - \underbrace{\E[f_n \mid \mc B]}_{= f_n} \|_{L^1(\lambda)}= 1.$$
So we cannot hope for uniform convergence over such a large class of functions.
\end{ex}

Instead of uniform convergence over all $f$ with $\| f \|_{L^\infty(\mu)} \leq 1$, we consider uniform convergence over 1-Lipschitz $f$. We again use the coordinate-wise dominated boxes $A_x \coloneqq [0,x_1] \times \cdots \times [0,x_d]$, but this choice is arbitrary, and one can prove uniform convergence with other choices for $A_x$ (perhaps with different rates).

Due to the asymmetrical nature of this partition, the sets containing points with coordinates near $1$ will be larger the the sets containing coordinates near 0, leading to a slow rate of convergence of $O((\log n/n)^{1/d})$. After stating this slow rate, we will see that a symmetrizing adjustment to this partition leads to a much faster rate of $O(1/n)$.

\begin{restatable}{thm}{uniformrates} \label{uniform rates}
Let $([0,1]^d,\mc B,\mu)$ be a probability space equipped with the Borel $\sigma$-field, and let $X_1,X_2,\dots \overset{\on{iid}}{\sim} \mu$, where $\mu \ll \lambda$ and $\gamma^{-1} < \frac{d\mu}{d\lambda} < \gamma$ for some $\gamma \geq 1$. For $x = (x_1,\dots,x_d) \in [0,1]^d$, let $A_x \coloneqq [0,x_1] \times \cdots \times [0,x_d]$, and define the empirical $\sigma$-fields $\mc F_n \coloneqq \sigma(A_{X_1},\dots,A_{X_n})$. Then
$$\sup_{\| f \|_{\on{Lip}}\leq 1} \| \E[f \mid \mc F_n] - f \|_{L^1(\mu)} \xrightarrow{\P\text{-a.s.},L^1(\P)} 0,$$
where $\| f \|_{\on{Lip}}: = \sup \{ \frac{|f(x) - f(y)|}{|x-y|} : x \neq y\}$ is the Lipschitz norm. Moreover,
$$\E \squa{\sup_{\| f \|_{\on{Lip}}\leq 1} \| \E[f \mid \mc F_n] - f \|_{L^1(\mu)}} \ls \paren{\frac{\log n}{n}}^{1/d} \qquad \forall n \geq 3.$$
The constant factor in the bound depends only on $d$ and $\gamma$.
\end{restatable}

The proof is in Appendix~\ref{uniform_appendix}.

To improve the convergence, we use the more symmetric partition $\wt {\mc F_n} \coloneqq \sigma( \{ x : x_i \leq X_{j,i}  \} : 1 \leq i \leq d, 1 \leq j \leq n )$, which splits the unit cube in two pieces along every coordinate of each sample point $X_j$. See Figure~\ref{symmetrized_partition} for an illustration.

\begin{figure}[h]
\begin{center}
\begin{tikzpicture}

  \draw[thick] (0,0) rectangle (5,5);

  \coordinate (P1) at (1.2, 3.8); 
  \coordinate (P2) at (3.5, 1.6); 
  \coordinate (P3) at (2.8, 4.2); 

  \draw[ blue] (1.2, 0) -- (1.2, 5); 
  \draw[ blue] (0, 3.8) -- (5, 3.8); 

  \draw[ red] (3.5, 0) -- (3.5, 5); 
  \draw[ red] (0, 1.6) -- (5, 1.6); 

  \draw[ gray] (2.8, 0) -- (2.8, 5); 
  \draw[ gray] (0, 4.2) -- (5, 4.2); 
  
  \filldraw[black] (P1) circle (2pt);
  \filldraw[black] (P2) circle (2pt);
  \filldraw[black] (P3) circle (2pt);

\end{tikzpicture}
\end{center}
\caption{The points in $\wt{\mc F_n}$ splitting the unit cube in every coordinate.}
\label{symmetrized_partition}
\end{figure}
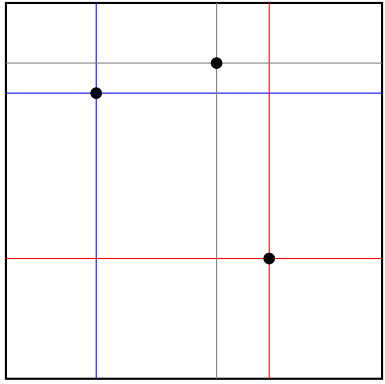

Now, we get a much faster rate:

\begin{restatable}[Faster uniform convergence with symmetrized $A_x$]{thm}{uniformratesfaster} \label{uniform rates faster}
Let $([0,1]^d,\mc B,\mu)$ be a probability space equipped with the Borel $\sigma$-field, and let $X_1,X_2,\dots \overset{\on{iid}}{\sim} \mu$, where $\mu \ll \lambda$ and $\gamma^{-1} < \frac{d\mu}{d\lambda} < \gamma$ for some $\gamma \geq 1$. Define the empirical $\sigma$-fields $\wt {\mc F_n} \coloneqq \sigma( \{ x : x_i \leq X_{j,i}  \} : 1 \leq i \leq d, 1 \leq j \leq n )$. Then
$$\sup_{\| f \|_{\on{Lip}}\leq 1} \| \E[f \mid \wt{\mc F_n}] - f \|_{L^1(\mu)} \xrightarrow{\P\text{-a.s.},L^1(\P)} 0,$$
where $\| f \|_{\on{Lip}}: = \sup \{ \frac{|f(x) - f(y)|}{|x-y|} : x \neq y\}$ is the Lipschitz norm. Moreover,
$$\E \squa{\sup_{\| f \|_{\on{Lip}}\leq 1} \| \E[f \mid \wt{\mc F_n}] - f \|_{L^1(\mu)}} \ls \frac{\sqrt d}{n} \qquad \forall n \geq 1.$$
The constant factor in the bound depends only on $\gamma$.
\end{restatable}

The proof is in Appendix~\ref{uniform_appendix}.

\begin{rem}
By scaling the sides of the box by constants, the results in Theorem~\ref{uniform rates} and Theorem~\ref{uniform rates faster} apply to boxes in $\R^d$ which are not $[0,1]^d$. We incur only an extra multiplicative factor of the volume of the box in our bound. Similarly, if we allow $f$ to be $L$-Lipschitz, we incur only a factor of $L$.
\end{rem}

\begin{rem}
The bound in Theorem~\ref{uniform rates faster} is tight. For a lower bound, consider the example $f(x) = x_1$ and $\mu = \lambda$. The partition is an axis-aligned grid, so the conditional expectation of $f$ in any set in the box is just the average of the maximal and minimal $x_1$ values for that box. So the integral is independent of the latter $d-1$ coordinates, and the problem reduces to a 1-dimensional problem.

Denoting the 1st coordinate of each $X_j$ as $X_{1,1},X_{2,1},\dots,X_{n,1}$ and denoting the order statistics of these values as $0 = Y_0 < Y_1 < \cdots < Y_n < Y_{n+1} = 1$, we write
\begin{align*}
\| \E[f \mid \wt{\mc F_n}] - f \|_{L^1(\mu)} &= \sum_{k=0}^n \int_{Y_k}^{Y_{k+1}} \abs{\frac{Y_k + Y_{k+1}}{2} - x_1} \, dx_1 \\
&= \sum_{k=0}^n \frac{(Y_{k+1} - Y_k)^2}{4},
\end{align*}
Taking expectations, we get
\begin{align*}
\E[\| \E[f \mid \mc F_n] - f \|_{L^1(\mu)} ]  &\geq \E[\| \E[f \mid \wt{\mc F_n}] - f \|_{L^1(\mu)} ] \\
&= \frac{1}{4} \sum_{k=0}^n \E[(Y_{k+1} - Y_k)^2],
\shortintertext{Where the $Y_k$ are the order statistics of $n$ iid uniform random variables on $[0,1]$. The differences of these successive order statistics are Beta($1$,$n$) distributed, so this equals}
&= \frac{1}{4} (n+1) \frac{2}{(n+1)(n+2)} \\
&= \frac{1}{2(n+2)}.
\end{align*}
\end{rem}

The $\sigma$-field $\wt{\mc F_n}$ is a refinement of $\mc F_n$, so the lower bound of $1/n$ applies to $\mc F_n$, as well.

\section{Applications} \label{app-section}

\subsection{Randomized Skorokhod embeddings}

Skorokhod (\cite{MR0185619}, translated into English \cite{MR0185620}) posed and solved the problem of embedding distributions of real-valued random variables into Brownian motion by stopping the process at suitably constructed random times. Since then, many solutions to the Skorokhod embedding problem have been discovered, with varying properties of interest; see \cite{obloj2004} for a survey detailing the various constructions and their historical context and \cite{beiglbock2017optimal} for a more recent work unifying many solutions to the problem.

Of particular note for our purposes is Dubins' 1968 solution to the Skorokhod embedding problem \cite{dubins1968theorem}. By adjusting Dubins' solution, we will provide a method of \emph{randomly generating} Skorokhod embeddings for a given distribution $\mu$.

Dubins' construction proceeds via a \emph{binary splitting martingale}. Suppose $X \sim \mu$ (with $\E[X] = 0$) and we want to generate the distribution of $X$ via a stopping time $T$ for Brownian motion (meaning $B_T \overset{\on{d}}{=} X$). We first create barriers for the Brownian motion at the points $x_1 \coloneqq \mathbb E[X \mid X < 0]$ and $x_2 \coloneqq \E[X \mid X > 0]$ and let $T_1 \coloneqq \inf \{ t > 0 : B_t \in  \{x_1,x_2 \} \}$. This divides the line into four intervals.

\begin{figure}[h]
\centering
\begin{tikzpicture}[scale=1.1]

  \draw[<->] (-5,2) -- (5,2);
  
  \draw[thick] (0,2.2) -- (0,1.8);
  \node[below] at (0,1.8) {$0$};

  \draw[thick, red] (-2.5,2.2) -- (-2.5,1.8);
  \node[below] at (-2.5,1.8) {$x_1$};

  \draw[thick, red] (2.5,2.2) -- (2.5,1.8);
  \node[below] at (2.5,1.8) {$x_2$};

  \draw[decorate,decoration={brace,mirror},yshift=-0.8cm]
    (-5,2) -- (-2.55,2) node[midway,below=5pt] {\footnotesize $X < x_1$};
    
  \draw[decorate,decoration={brace,mirror},yshift=-0.8cm]
    (-2.45,2) -- (-0.05,2) node[midway,below=5pt] {\footnotesize $x_1 < X < 0$};
    
  \draw[decorate,decoration={brace,mirror},yshift=-0.8cm]
    (0.05,2) -- (2.45,2) node[midway,below=5pt] {\footnotesize $0 < X < x_2$};
    
  \draw[decorate,decoration={brace,mirror},yshift=-0.8cm]
    (2.55,2) -- (5,2) node[midway,below=5pt] {\footnotesize $X > x_2$};

  \draw[<->] (-5,0) -- (5,0);

  \draw[thick] (0,0.2) -- (0,-0.2);
  \node[below] at (0,-0.2) {$0$};

  \draw[thick, red] (-2.5,0.2) -- (-2.5,-0.2);
  \node[below] at (-2.5,-0.2) {$x_1$};

  \draw[thick, red] (2.5,0.2) -- (2.5,-0.2);
  \node[below] at (2.5,-0.2) {$x_2$};

  \draw[thick, blue] (-4,0.125) -- (-4,-0.125);
  \node[below] at (-4,-0.2) {$y_1$};

  \draw[thick, blue] (-1.5,0.125) -- (-1.5,-0.125);
  \node[below] at (-1.5,-0.2) {$y_2$};

  \draw[thick, blue] (1.5,0.125) -- (1.5,-0.125);
  \node[below] at (1.5,-0.2) {$y_3$};

  \draw[thick, blue] (4,0.125) -- (4,-0.125);
  \node[below] at (4,-0.2) {$y_4$};

  \draw[decorate,decoration={brace,mirror},yshift=-0.8cm]
    (-5,0) -- (-2.55,0) node[midway,below=6pt] {\footnotesize $X < x_1$};
    
  \draw[decorate,decoration={brace,mirror},yshift=-0.8cm]
    (-2.45,0) -- (-0.05,0) node[midway,below=6pt] {\footnotesize $x_1 < X < 0$};
    
  \draw[decorate,decoration={brace,mirror},yshift=-0.8cm]
    (0.05,0) -- (2.45,0) node[midway,below=6pt] {\footnotesize $0 < X < x_2$};
    
  \draw[decorate,decoration={brace,mirror},yshift=-0.8cm]
    (2.55,0) -- (5,0) node[midway,below=6pt] {\footnotesize $X > x_2$};

\end{tikzpicture}
\caption{Top: first step of Dubins' binary splitting, with barriers $x_1$ and $x_2$. Bottom: refinement using $y_1,\dots,y_4$.}
\label{dubins-interval}
\end{figure}
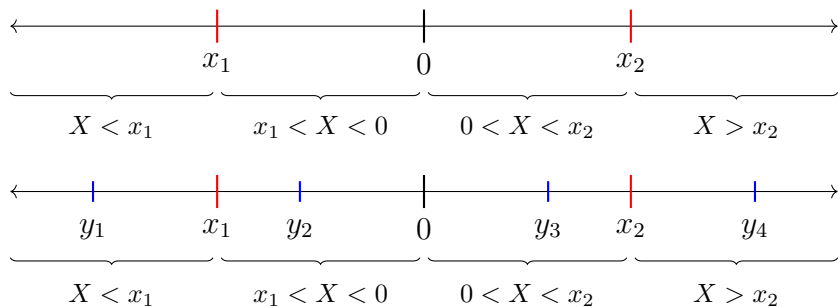

For the next step, we divide each of the intervals in two by adding more barriers. In particular, we add barriers
$$y_1 \coloneqq \E[X \mid X \leq x_1], \qquad y_2 \coloneqq \E[X \mid x_1 < X < 0],$$
$$y_3 \coloneqq \E[X \mid 0 < X \leq x_2], \qquad y_4 \coloneqq \E[X \mid X > x_2]$$
and let $T_2 \coloneqq \inf \{ t > T_1 : B_t \in  \{y_1,y_2, y_3,y_4 \} \}$. See Figure~\ref{dubins-interval} for an illustration.

Repeating this process, we end up with a sequence of stopping times $(T_n)_{n=1}^\infty$ for Brownian motion such that $B_{T_n}$ equals, with equal probability, any of the $2^n$ level $n$ barrier points. In fact, a more careful analysis of this process shows that $B_{T_n} \overset{d} {=} \E[X \mid \mc B_n]$, where $\mc B_n$ is the $\sigma$-field representing the partition of the interval by all barrier points up to level $n$. Taking $T \coloneqq \lim_{n \to \infty} T_n$ gives us the stopping time we desire. Figure~\ref{dubins-brownian} illustrates the first few steps of this process on a simulated Brownian motion.

\begin{figure}[h]
\begin{center}
\begin{tikzpicture}[ >=Stealth,scale=1.25]

  \draw[<->] (0,-3) -- (0,3);
  \draw[->] (0,0) -- (7,0);

  \draw[dashed, thick, red] (0,-1) -- (2.1,-1);
  \node[left] at (0,-1) {$x_1$};

  \draw[dashed, thick, red] (0,1) -- (2.1,1);
  \node[left] at (0,1) {$x_2$};

  \draw[dashed, thick, red] (2.1,-2) -- (4.75,-2);
  \node[left] at (0,-2) {$y_1$};

  \draw[dashed, thick, red] (2.1,-0.5) -- (4.75,-0.5);
  \node[left] at (0,-0.5) {$y_2$};

  \draw[dashed, thick, red] (2.1,0.5) -- (4.75,0.5);
  \node[left] at (0,0.5) {$y_3$};

  \draw[dashed, thick, red] (2.1,2) -- (4.75,2);
  \node[left] at (0,2) {$y_4$};

  \draw[dashed, thick, red] (4.75,2.5) -- (5.12,2.5);

  \draw[dashed, thick, red] (4.75,1.7) -- (5.12,1.7);

  \draw[dashed, thick, red] (4.75,0.85) -- (5.12,0.85);
  
  \draw[dashed, thick, red] (4.75,0.25) -- (5.12,0.25);
  
  \draw[dashed, thick, red] (4.75,-0.15) -- (5.12,-0.15);

  \draw[dashed, thick, red] (4.75,-0.75) -- (5.12,-0.75);

  \draw[dashed, thick, red] (4.75,-1.3) -- (5.12,-1.3);
  
  \draw[dashed, thick, red] (4.75,-2.3) -- (5.12,-2.3);

  \draw[dashed, thick] (2.1,-3) -- (2.1,3); 
  \node[above] at (2.1,3) {$T_1$};

  \draw[dashed, thick] (4.75,-3) -- (4.75,3); 
  \node[above] at (4.75,3) {$T_2$};
  
  \draw[dashed, thick] (5.12,-3) -- (5.12,3); 
  \node[above] at (5.12,3) {$T_3$};

  \def\nsteps{200} 
  \def\xmax{7}
  \pgfmathsetmacro{\dx}{\xmax/\nsteps}
  \def\scaleY{0.2} 
  \def\std{0.01} 
  
  \pgfmathsetseed{777} 

  \coordinate (p0) at (0,0);
  \pgfmathsetmacro{\lasty}{0}

  \foreach \i in {1,...,\nsteps} {
    \pgfmathsetmacro{\dy}{50*randnormal*\std}
    \pgfmathsetmacro{\thisy}{\lasty+\dy}
    \pgfmathsetmacro{\thisx}{\i*\dx}
    \coordinate (p\i) at (\thisx,{\thisy*\scaleY}); 
    \global\let\lasty\thisy
  }

  \draw[blue, thick]
    (p0)
    \foreach \i in {1,...,\nsteps} {
      -- (p\i)
    };

\end{tikzpicture}
\end{center}
\caption{The first 3 steps of stopping times in Dubins' construction.}
\label{dubins-brownian}
\end{figure}
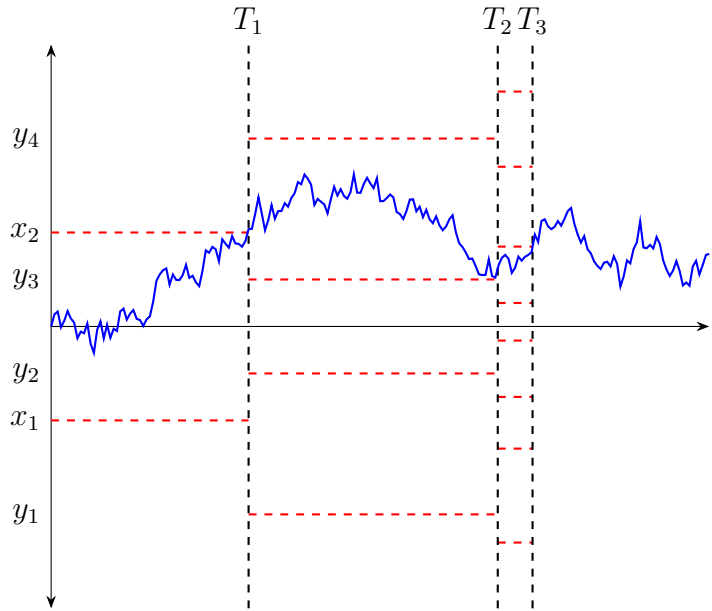

The key insight for this application of our framework is that Dubins' meticulously constructed ``dyadic'' partitions of the line are not actually necessary. We will show that any (deterministic) sequence of partitions adding 1 point at a time suffices for the embedding, provided that the information resolution of the partitions asymptotically captures the degree of resolution associated to $\mu$. Applying our framework in the context of generating random partitions from iid sampling, we obtain random Skorokhod embeddings.

The following theorem constructs a Skorokhod embedding for a (deterministic) sequence of partitions.

\begin{thm} \label{partition-skorokhod}
Let $\mu$ be a distribution on $\R$ with mean zero and finite second moment, and let $X \sim \mu$. Let $(x_n)_{n=1}^\infty$ be a sequence of real numbers, and let $\mc F_n \coloneqq \sigma((-\infty,x_1],\dots,(-\infty,x_n])$ define a filtration such that $\mc F_n \uparrow \mc B$, i.e.\ $\bigvee_{n=1}^\infty \mc F_n$ and the Borel $\sigma$-field $\mc B$ differ only by $\mu$-null sets. There exists a stopping time $T(x_1,x_2,\dots)$ for Brownian motion such that, $\P$-a.s., $B_T \overset{d}{=} X$ and $\E[T] = \E[X^2]$.
\end{thm}


\begin{proof}
Define a sequence of stopping times by $T_0 = 0$ and $T_{n+1} = \inf \{ t > T_n : B_t \in \on{ran}(\E[X \mid \mc F_n]) \}$. Then $T_0 \leq T_1 \leq T_2 \leq \cdots$, so there exists a (possibly infinite) stopping time $T = \lim_{n \to \infty} T_n$. Moreover, $B_{T_n} \overset{d}{=} \E[X \mid \mc F_n]$ for each $n$, as $B_{T_{n+1}} \mid B_{T_n} = x$ is equal to or supported on the same two points as $\E[X \mid \mc F_{n+1}] \mid \E[X \mid \mc F_n] = x$, and
$$\E[B_{T_{n+1}} \mid B_{T_n} = x] = x = \E[\E[X \mid \mc F_{n+1}] \mid \E[X \mid \mc F_n] = x].$$
The latter equality is due to the fact that $\E[\E[X \mid \mc F_{n+1}] \mid \mc F_n] = \E[X \mid \mc F_n]$.

This lets us bound the size of $T_n$, as
$$\E[T_n] = \E[\E[B_{T_n}^2 \mid T_n]] = \E[B_{T_n}^2] = \E[(\E[X \mid \mc F_n])^2] \leq \E[X^2],$$
where we have used the tower property of conditional expectation and the conditional version of Jensen's inequality. By the monotone convergence theorem, $\E[T] \leq \E[X^2]$, from which we conclude that $T < \infty$ a.s. Now, by Theorem~\ref{classical R^n G-C} and the martingale convergence theorem, $\E[X \mid \mc F_n]$ converges in distribution to $X$. By the continuity of Brownian motion paths, $B_{T_n}$ converges in distribution to $B_T$. Thus, we may conclude that $B_T \overset{d}{=} X$, from which we conclude
\begin{equation*}
\E[T] = \E[\E[B_T^2 \mid T]] = \E[B_T^2] = \E[X^2].\qedhere
\end{equation*}
\end{proof}

\begin{cor}[Randomized Skorokhod embedding] \label{randomized-skorokhod}
Let $\mu$ be a distribution on $\R$ with mean zero and finite second moment, and let $X, X_1,X_2,\dots \overset{\on{iid}}{\sim} \mu$. There exists a randomized (depending on $X_1,X_2,\dots$) stopping time $T$ for Brownian motion such that, $\P$-a.s., $B_T \overset{d}{=} X$ and $\E[T \mid X_1,X_2,\dots] = \E[X^2]$.
\end{cor}

\begin{proof}
We apply Theorem~\ref{partition-skorokhod} to the sequence of empirical $\sigma$-fields given by $\mc F_n \coloneqq \sigma((-\infty,X_1], \dots, (-\infty,X_n])$. Theorem~\ref{classical R^n G-C} shows that $\mc F_n \uparrow \mc B$.
\end{proof}

\begin{rem}
It is not necessary for $X_1,X_2,\dots$ to be sampled from the same measure as $X$. Theorem~\ref{randomized-skorokhod} still holds if we sample $X_1,X_2,\dots \overset{\on{iid}}{\sim} \nu$, provided that $\supp \nu \supseteq \supp \mu$. This has the interesting consequence that there exist \emph{universal} generating measures for randomized Skorokhod embeddings. For example, if $\nu$ is the standard normal distribution (or any other measure with full support), then sampling $X_1,X_2,\dots \overset{\on{iid}}{\sim} \nu$ generates a randomized Skorokhod embedding construction which is valid for any $\mu$.
\end{rem}

This construction yields different Skorokhod embeddings for each sequence of values $X_1,X_2,\dots$. See Figure~\ref{skorokhod-sim} for a simulation comparing Dubins' classical Skorokhod embedding and two independent randomized Skorokhod embeddings on the same Brownian motion.

\begin{figure}[h]
\begin{center}
\includegraphics[scale = 0.75]{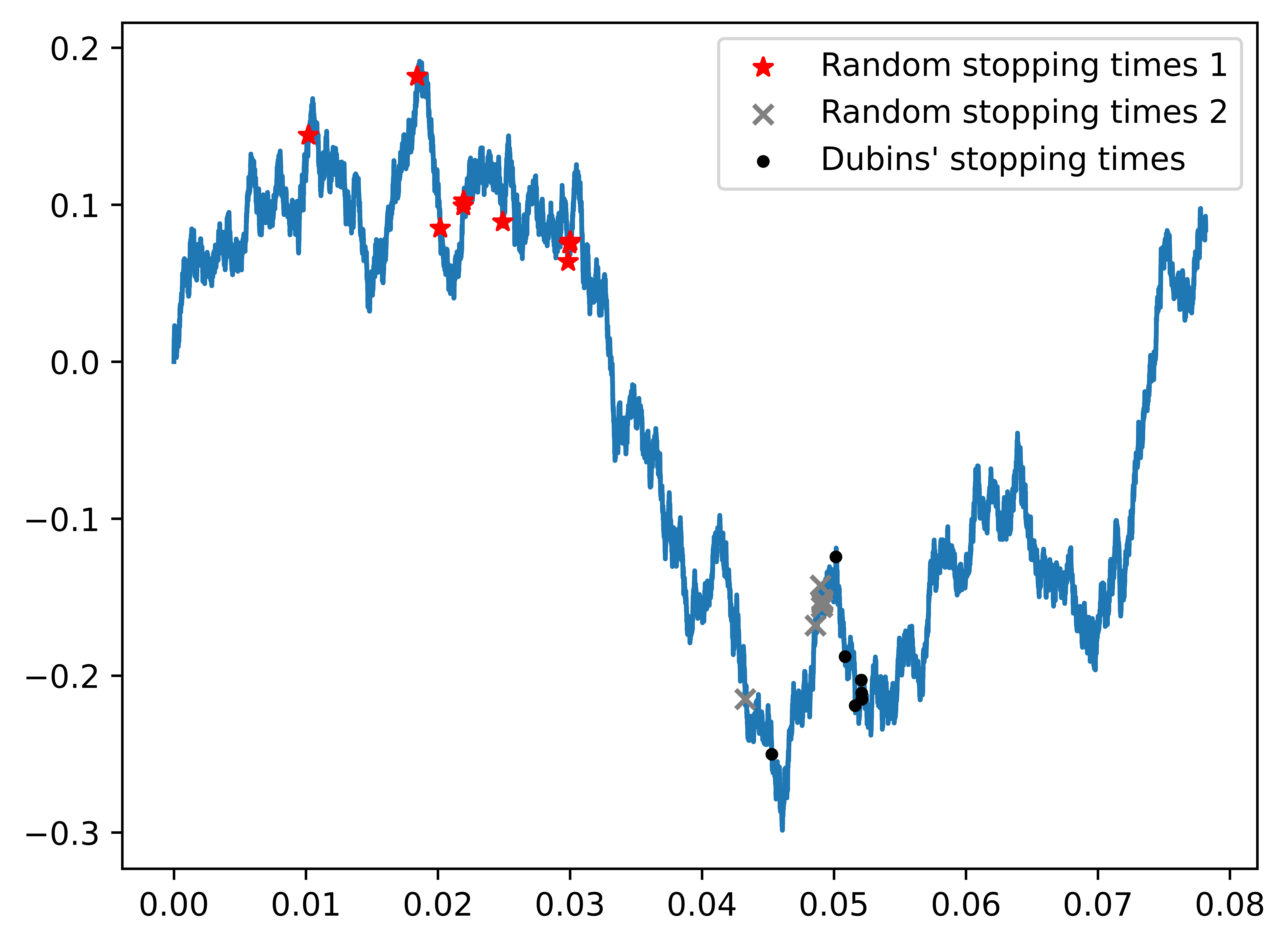}
\end{center}
\caption{Stopping times for Dubins' embedding and two independent randomized embeddings on the same Brownian motion. Here, we are embedding the uniform distribution on $[-0.5,0.5]$.}
\label{skorokhod-sim}
\end{figure}

\subsection{Random splitting random forests}

Our second example application of this framework is to obtain uniform risk bounds for randomized regression trees in a random forest. Random forest models \cite{breiman2001random} are popular machine learning tools for tasks such as classification and regression. In the case of regression, the model constructs a number of regression trees, with splits determined by some optimal choice of splitting along a randomly selected subset of the feature coordinates; see Figure~\ref{regression-splits} for an illustration of splitting the feature space. Then, within each box of the feature space, the model reports the average of the values of the data points in that box.

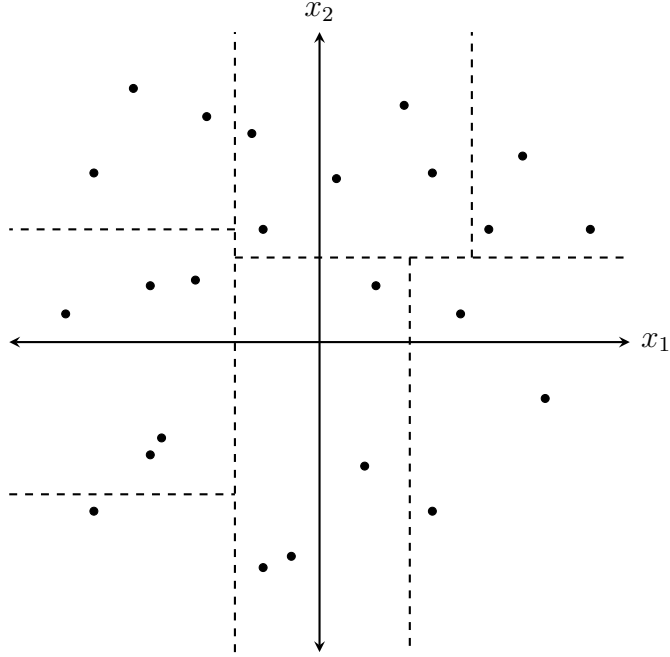
\begin{figure}[h]
\begin{center}
\begin{tikzpicture}[scale = 0.75]

  \draw[<->, thick, >=stealth] (-5.5,0) -- (5.5,0) node[right] {$x_1$};
  \draw[<->, thick, >=stealth] (0,-5.5) -- (0,5.5) node[above] {$x_2$};

  \foreach \x/\y in {
    -4/3, -3/-2, -2/4, -1/-4, 1/1, 2/-3, 3/2, 4/-1, 
    -4/-3, -1/2, 2/3, -3/1, 1.5/4.2, -2.8/-1.7, 
    3.6/3.3, -4.5/0.5, 0.8/-2.2, -1.2/3.7, 2.5/0.5, 
    -3.3/4.5, 4.8/2, -0.5/-3.8, -2.2/1.1, 0.3/2.9
  } {
    \filldraw[black] (\x,\y) circle (2pt);
  }

  \draw[dashed, thick] (-1.5,-5.5) -- (-1.5,5.5);
  \draw[dashed, thick] (2.7,1.5) -- (2.7,5.5);
  \draw[dashed, thick] (1.6,1.5) -- (1.6,-5.5);

  \draw[dashed, thick] (-5.5,-2.7) -- (-1.5,-2.7);
  \draw[dashed, thick] (-1.5,1.5) -- (5.5,1.5);
  \draw[dashed, thick] (-1.5,2) -- (-5.5,2);

\end{tikzpicture}
\end{center}
\caption{Axis-parallel splits of the feature space in a regression tree.}
\label{regression-splits}
\end{figure}

The key facet relating regression trees to our considerations is that a regression tree is essentially reporting the conditional expectation with respect to a partition of the feature space. From this perspective, we build our tree by refining the partition, i.e.\ by increasing the resolution of the associated $\sigma$-field. So we can study the error incurred in building our tree via convergence of the $\sigma$-fields representing these partitions.

For a regression tree, even with an infinite amount of data, performance is bottlenecked by the coarseness of the resolution. Here, we use the notion of information resolution to address the following question: given infinite data, how does the error decay as the resolution becomes finer? While we focus on the infinite-data setting for simplicity, similar ideas could be used to study the trade-off between sample size and resolution.

We can alter the standard random forest model by constructing regression trees using \emph{random splits}, similarly to the Extra-Trees algorithm from \cite{cite-key}. That is, we pick random points $G_1,\dots,G_m \overset{\on{iid}}{\sim} \nu$ and construct a partition from these points. For example, we could construct a grid using all axis-parallel lines passing through $G_1,\dots,G_m$, or we could use an asymmetric partition such as the one in Theorem~\ref{uniform rates}; Figure~\ref{random-splits} illustrates this variant of random splits.

\begin{figure}[h]
\begin{center}
\begin{tikzpicture}[scale = 0.75]

  \draw[<->, thick, >=stealth] (-5.5,0) -- (5.5,0) node[right] {$x_1$};
  \draw[<->, thick, >=stealth] (0,-5.5) -- (0,5.5) node[above] {$x_2$};

  \foreach \x/\y in {
    -4/3, -3/-2, -2/4, -1/-4, 1/1, 2/-3, 3/2, 4/-1, 
    -4/-3, -1/2, 2/3, -3/1, 1.5/4.2, -2.8/-1.7, 
    3.6/3.3, -4.5/0.5, 0.8/-2.2, -1.2/3.7, 2.5/0.5, 
    -3.3/4.5, 4.8/2, -0.5/-3.8, -2.2/1.1, 0.3/2.9
  } {
    \filldraw[black] (\x,\y) circle (2pt);
  }
  
  \draw[dashed, thick] (1.2,0.7) -- (-5.5,0.7);
  \draw[dashed, thick] (1.2,0.7) -- (1.2,-5.5);
  
  \filldraw[gray](1.2,0.7) circle (3pt);
  \draw[black](1.2,0.7) circle (3pt);

  \draw[dashed, thick] (1.8,4.3) -- (-5.5,4.3);
  \draw[dashed, thick] (1.8,4.3) -- (1.8,-5.5);
  
  \filldraw[gray](1.8,4.3) circle (3pt);
  \draw[black](1.8,4.3) circle (3pt);

  \draw[dashed, thick] (5.1,-2.8) -- (-5.5,-2.8);
  \draw[dashed, thick] (5.1,-2.8) -- (5.1,-5.5);
  
  \filldraw[gray](5.1,-2.8) circle (3pt);
  \draw[black](5.1,-2.8) circle (3pt);

  \draw[dashed, thick] (-2,3.2) -- (-5.5,3.2);
  \draw[dashed, thick] (-2,3.2) -- (-2,-5.5);
  
  \filldraw[gray](-2,3.2) circle (3pt);
  \draw[black](-2,3.2) circle (3pt);

\end{tikzpicture}
\end{center}
\caption{Splitting the feature space using random points for an asymmetric partition.}
\label{random-splits}
\end{figure}
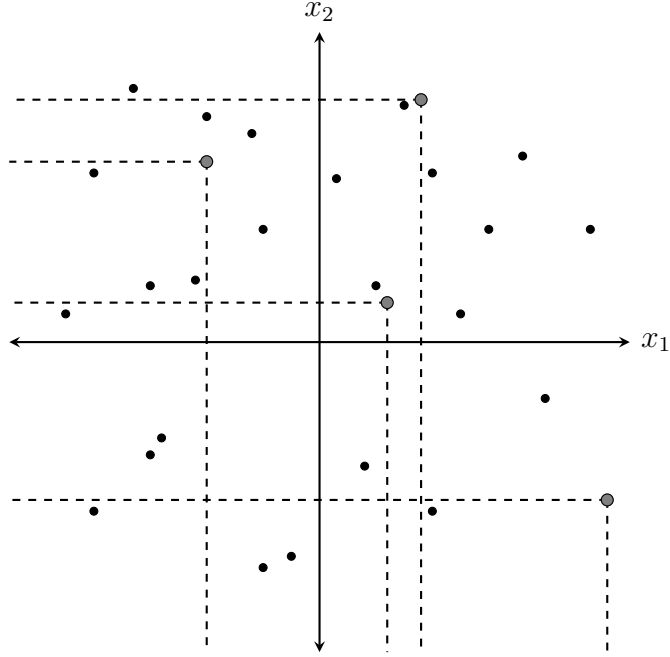

In this setting, our Theorem~\ref{uniform rates} essentially immediately provides a bound on the risk, where the parameter $f$ can even be chosen adversarially against our regression tree estimator. For simplicity, we will treat the case of the partitions from Theorem~\ref{uniform rates} and Theorem~\ref{uniform rates faster}, but the same analysis could be carried out with other choices of randomized sets $A_y$.

\begin{thm}[Random splitting regression tree loss] \label{tree loss thm}
Let $(X_i,Y_i)_{i=1}^N$ be drawn iid according to $Y = f(X) + \varepsilon$, where $\varepsilon$ is independent of $X$ with $\E[\varepsilon] = 0$ and $\Var(\varepsilon) = \sigma^2$. Draw $(G_k)_{1 \leq k \leq m} \overset{\on{iid}}{\sim} \nu$ with $\gamma^{-1} < \frac{d\nu}{d\lambda} < \gamma$ for some $\gamma \geq 1$, define $\mc F_m \coloneqq \sigma(A_{G_{1}},\dots,A_{G_{m}})$ with $A_y \coloneqq \{ x : x_i \leq y_i \: \forall 1 \leq i \leq d\}$, and define the random splitting regression tree estimator
$$\wh f(x) \coloneqq \frac{1}{|R_x|} \sum_{i : X_i \in R_x} Y_i,$$
where $R_x$ is the set containing $x$ in the finest partition given by $\mc F_m$. Then
$$\limsup_{N \to \infty} \sup_{\| f \|_{\on{Lip}}\leq 1} \E \squa{\| \wh f - f \|_{L^1(\mu)}} \ls\paren{\frac{\log m}{m}}^{1/d}.$$
If we use $\wt {\mc F_m} \coloneqq \sigma(\{ x : x_i \leq X_{j,i}\} : 1 \leq i \leq d, 1 \leq j \leq m  )$ in place of $\mc F_m$, then 
$$\limsup_{N \to \infty} \sup_{\| f \|_{\on{Lip}}\leq 1} \E \squa{\| \wh f - f \|_{L^1(\mu)}} \ls \frac{\sqrt d}{m}.$$
\end{thm}

\begin{rem}
We only take the limit as $N \to \infty$ (infinitely many samples) to guarantee that every set in the partition of the feature space a.s.\ contains at least 1 data point (so that $\wh f$ is well-defined). Depending on the choice of sets $A_y$ (and their ensuing geometry), one may calculate the relationship between $N$ and $m$ to ensure that with high probability, no partition set is empty.
\end{rem}

\begin{proof}
We will treat the case of $\mc F_m$; the proof for $\wt{\mc F_m}$ is similar. In taking the limit as $N \to \infty$, we may assume that all grid boxes contain at least one $X_i$, so that $\wh f$ is well-defined. Then, using the triangle inequality,
\begin{align*}
\limsup_{N \to \infty} \sup_{\| f \|_{\on{Lip}}\leq 1} \E \squa{ \| \wh f - f \|_{L^1(\mu)}} &\leq \limsup_{N \to \infty} \sup_{\| f \|_{\on{Lip}}\leq 1} \E \squa{ \| \wh f - \E[f \mid \mc F_m] \|_{L^1(\mu)}} \\
&\qquad + \limsup_{N \to \infty} \sup_{\| f \|_{\on{Lip}}\leq 1} \E \squa{\| \E[f \mid \mc F_m] - f \|_{L^1(\mu)}}
\end{align*}
Theorem~\ref{uniform rates} upper bounds the latter term by $O \paren{\paren{\frac{\log m}{m}}^{1/d}}$. The former term can be controlled by noting that for any $f$ with $\| f \|_{\on{Lip}}\leq 1$,
\begin{align*}
\| \wh f - \E[f \mid \mc F_m] \|_{L^1(\mu)} &\leq \int \frac{1}{|R_x|} \sum_{i : X_i \in R_x} \abs{f(X_i) - \frac{1}{\mu(R_x)} \int_{R_x} f(y) \, d\mu(y)} \, d\mu(x) \\
&\leq \int \frac{1}{|R_x|} \sum_{i : X_i \in R_x} \frac{1}{\mu(R_x)} \int_{R_x} |f(X_i) - f(y)| \, d\mu(y) \, d\mu(x) \\
&\leq \int \diam(R_x) \, d\mu(x) \\
&= \sum_{R \in \mc P_m} \mu(R) \diam(R),
\end{align*}
where $\mc P_m$ denotes the finest partition given by $\mc F_m$. Bounding this quantity as in the proof of Theorem~\ref{uniform rates}, we get that the second term is $O \paren{\paren{\frac{\log m}{m}}^{1/d}}$, as claimed.
\end{proof}

By averaging independently randomized regression trees, one may construct random forests without the need for bootstrap aggregation, optimizing the split points, or random selection of features. Figure~\ref{random-splitting-random-forest} compares the performance of such random splitting random forests (with 10 trees, using asymmetric and symmetric partitions) on the California housing dataset, originally introduced in \cite{KELLEYPACE1997291} and now available through the scikit-learn library, as the number of random splits increases. As predicted by Theorem~\ref{tree loss thm}, the symmetric partition requires vastly fewer random split points to make accurate predictions.

\begin{figure}[h]
\begin{center}
\includegraphics[scale = 0.75]{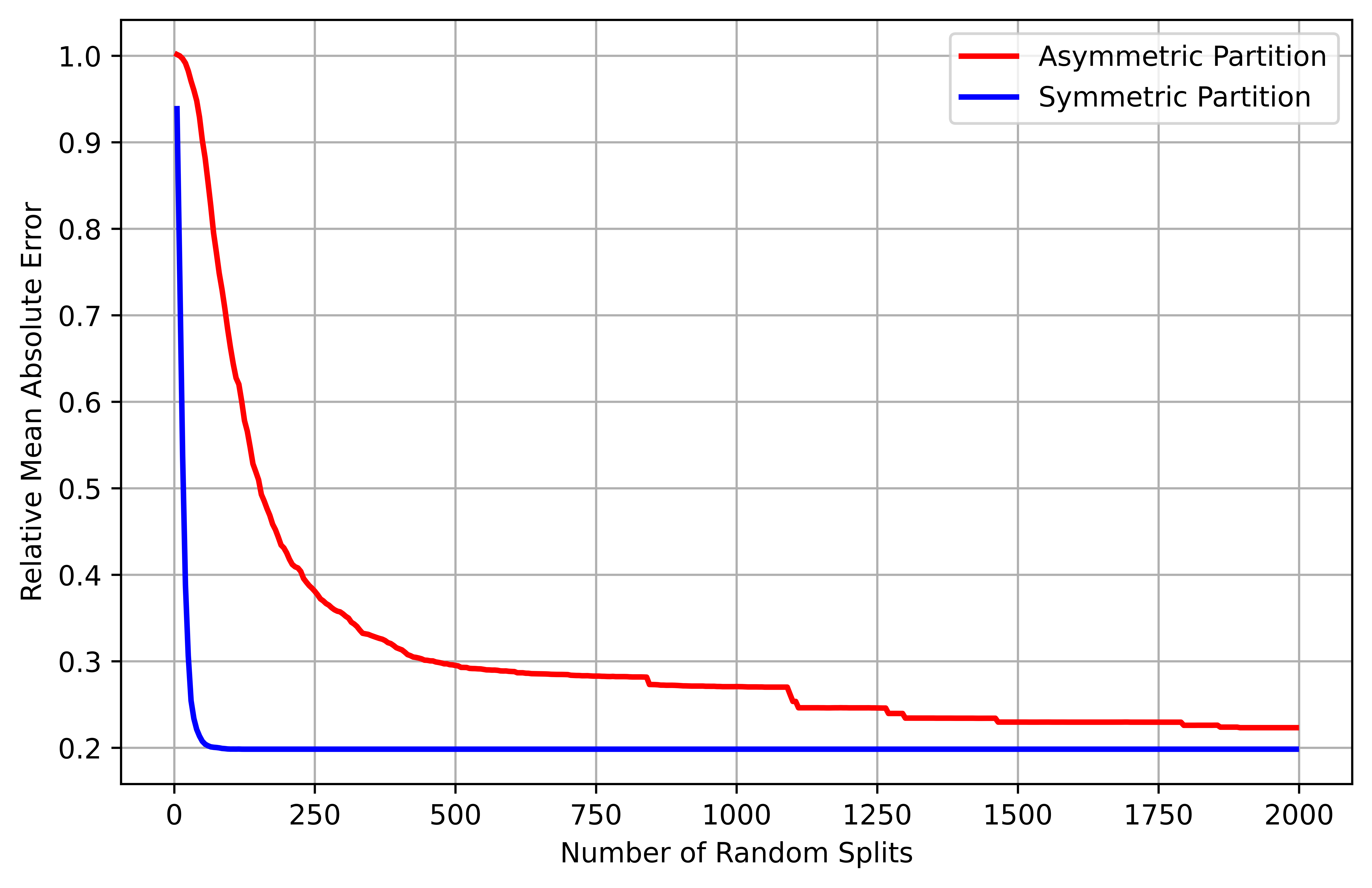}
\end{center}
\caption{Performance of asymmetric and symmetric random splitting random forests for predicting California housing prices.}
\label{random-splitting-random-forest}
\end{figure}

\paragraph*{Acknowledgements.} The author would like to thank Sinho Chewi, Steve Evans, Shirshendu Ganguly, Arvind Prasadan, and Jo\~{a}o Vitor Romano for many helpful comments and conversations. While writing this paper, the author was supported by a Two Sigma PhD Fellowship and a research contract with Sandia National Laboratories, a U.S.\ Department of Energy multimission laboratory.

\pagebreak


\bibliographystyle{alpha}
\bibliography{Information-Resolution-LLNs}

\appendix

\section{Proofs of Theorems~\ref{uniform rates} and \ref{uniform rates faster}} \label{uniform_appendix}

\uniformrates*

We first reduce the problem to the geometric problem of constructing a fine mesh partition of the support of $\mu$.

\begin{lem} \label{bound L^1 by diameter}
Fix the values of $X_1,X_2,\dots$, and denote $\mc P_n$ the finest partition given by the $\sigma$-field $\mc F_n$ (omitting any $\mu$-null sets). Then
$$\sup_{\| f \|_{\on{Lip}}\leq 1} \| \E[f \mid \mc F_n] - f \|_{L^1(\mu)} \leq \sum_{A \in \mc P_n} \mu(A) \diam(A),$$
where $\diam(A) \coloneqq \sup \{ |x-y| : x,y \in A\}$.
\end{lem}

\begin{proof}[Proof of Lemma~\ref{bound L^1 by diameter}]
\begin{align*}
\sup_{\| f \|_{\on{Lip}}\leq 1} \| \E[f \mid \mc F_n] - f \|_{L^1(\mu)} &= \sup_{\| f \|_{\on{Lip}}\leq 1} \int_{[0,1]^d} \abs{ \sum_{A \in \mc P_n} \E[ f \mid A] \mbbm 1_A(x) - f(x) } \, d\mu(x) \\
&\leq \sup_{\| f \|_{\on{Lip}}\leq 1}  \sum_{A \in \mc P_n} \int_{A} | \E[ f \mid A] - f(x)| \, d\mu(x) \\
&\leq \sup_{\| f \|_{\on{Lip}}\leq 1}  \sum_{A \in \mc P_n} \frac{1}{\mu(A)} \int_{A} \int_A |f(y) - f(x)| \, d\mu(y) \, d\mu(x) \\
&\leq \sum_{A \in \mc P_n} \mu(A) \diam(A).\qedhere
\end{align*}
\end{proof}

To bound the diameter, we use a slightly modified version of the approach taken for the proof of Lemma 40 in \cite{manole2021plugin}, which essentially uses a covering argument phrased in terms of Vapnik-Chervonenkis dimension.

\begin{proof}[Proof of Theorem~\ref{uniform rates}]
We first prove the $L^1(\P)$-convergence rate bound. Fix $0 < \delta < 1$, and consider a mesh dividing $[0,1]^d$ into cubes $C$ of side length $\varepsilon = (\frac{ \gamma \log(n/\delta)}{n})^{1/d}$. Then, with probability $\geq 1-\delta$, each cube in the mesh contains some sample point $X_i$ with $1 \leq i \leq n$ because
\begin{align*}
\P(\text{some cube has no samples}) &\leq \sum_{C}  \P(\text{$C$ has no samples}) \\
&= \sum_C (1 - \mu(C))^n \\
&\leq \sum_C (1 - \gamma^{-1} \varepsilon^d)^n \\
&= (1/\varepsilon)^d (1 - \gamma^{-1} \varepsilon^d)^n \\
&=  \frac{n}{\gamma \log(n/\delta)}  \paren{1 - \frac{ \log(n/\delta)}{n}}^n \\
&\leq  \frac{n}{\gamma \log(n/\delta)}  \exp (- \log(n/\delta)) \\
&= \frac{\delta}{\gamma \log(n/\delta)} \\
&\leq \delta.
\end{align*}
To upper bound $\sum_{A \in \mc P_n} \mu(A) \diam(A)$, first note that this quantity is monotonically nonincreasing in $n$, as splitting a set $A$ into multiple pieces cannot increase the diameter of either piece. So it suffices to show a bound on this quantity when we throw away all sample points $X_i$ except for one sample point in each mesh cube $C$. We will do so on the aforementioned probability $\geq 1-\delta$ event.

Excepting the set $L \in \mc P_n$ containing the point $(1,\dots,1)$, the diameter of any set $A \in \mc P_n \setminus \{ L \}$ must be $\leq \varepsilon \sqrt d $. The diameter of the corner set $L$ will be $\leq \sqrt d$ (the diameter of $[0,1]^d$), but on this event, $\lambda(L) \leq d \varepsilon$. Thus, we may bound
\begin{align*}
\sum_{A \in \mc P_n} \mu(A) \diam(A) &\leq \gamma \sum_{A \in \mc P_n} \lambda(A) \diam(A) \\
&= \gamma \lambda(L) \on{diam}(L) + \gamma \sum_{A \in \mc P_n \setminus \{L\}} \lambda(A) \diam(A) \\
&\leq \gamma d^{3/2} \varepsilon + 4 \gamma \varepsilon \sqrt d \sum_{A \in \mc P_n \setminus \{ L\}} \lambda(A) \\
&\leq 2 \gamma d^{3/2} \varepsilon.
\end{align*}
So, if we denote $K =  2 \gamma^{1+1/d} d^{3/2}$, using Lemma~\ref{bound L^1 by diameter} gives
\begin{align*}
\sup_{\| f \|_{\on{Lip}}\leq 1} \| \E[f \mid \mc F_n] - f \|_{L^1(\mu)} &\leq K \paren{\frac{ \log(n/\delta)}{n}}^{1/d} \\
\intertext{with probability $\geq 1- \delta$. Writing $\delta = n \exp(- \frac{u^d n}{K^d})$ for $ u > 0$,}
&= u.
\end{align*}
Thus, applying the argument over all $u > 0$ (that is, varying $\delta$ throughout $(0,1)$), we may estimate
\begin{align*}
\E \squa{\sup_{\| f \|_{\on{Lip}}\leq 1} \| \E[f \mid \mc F_n] - f \|_{L^1(\mu)}} &= \int_0^\infty \P \paren{\sup_{\| f \|_{\on{Lip}}\leq 1} \| \E[f \mid \mc F_n] - f \|_{L^1(\mu)} > u} \, du \\
\shortintertext{Picking a cutoff parameter $t_n = K (\frac{2 \log n}{dn})^{1/d}$,}
&\leq t_n + n\int_{t_n}^\infty \exp \paren{- \frac{u^d n}{K^d} } \, du
\shortintertext{Making the change of variables $v = u^{d/2}$,}
&= t_n + \frac{2n}{d} \int_{t_n^{d/2}}^\infty \exp \paren{- \frac{v^2 n}{K^d} } v^{2/d - 1} \, dv \\
&\ls t_n + n \int_{t_n^{d/2}}^\infty \exp \paren{- \frac{v^2 n}{2K^d} } v  \, dv \\
&= t_n + K^d\exp \paren{- \frac{t_n^d n}{2K^d} } \\
&= K \paren{\frac{2 \log n}{dn}}^{1/d} + \frac{K}{n^{1/d}} \\
&\ls \paren{\frac{\log n}{n}}^{1/d}.
\end{align*}

The $\P$-a.s.\ convergence follows from the $L^1(\P)$ convergence and the fact that $H_n \coloneqq \sum_{A \in \mc P_n} \mu(A) \diam(A)$ is nonincreasing in $n$ for each $\omega \in \Omega$. Indeed, if we let $E = \{ \omega \in \Omega : \limsup_{n \to \infty} H_n(\omega) > 0\}$ then
$$\limsup_{n \to \infty} \E[H_n \mbbm 1_E] \leq \limsup_{n \to \infty} \E[H_n] = 0.$$
But $\limsup_{n \to \infty} \E[H_n \mbbm 1_E] = \E[(\limsup_{n \to \infty}H_n) \mbbm 1_E]  > 0$ if $\P(E) > 0$, so we must have $\P(E) = 0$.
\end{proof}

\uniformratesfaster*

\begin{proof}
As before, it suffices to prove the expectation bound. Let $\mc P_n$ denote the finest partition given by the $\sigma$-field $\wt{\mc F_n}$ (omitting any $\mu$-null sets). We first bound this by the average distance between a point in $[0,1]^d$ and the upper right corner of the partition box it lies in. Then
\begin{align*}
\sup_{\| f \|_{\on{Lip}}\leq 1} \| \E[f \mid \wt{\mc F_n}] - f \|_{L^1(\mu)} &= \sup_{\| f \|_{\on{Lip}}\leq 1} \int_{[0,1]^d} \abs{ \sum_{A \in \mc P_n} \E[ f \mid A] \mbbm 1_A(x) - f(x) } \, d\mu(x) \\
&\leq \sup_{\| f \|_{\on{Lip}}\leq 1}  \sum_{A \in \mc P_n} \int_{A} | \E[ f \mid A] - f(x)| \, d\mu(x) \\
&\leq \sup_{\| f \|_{\on{Lip}}\leq 1}  \sum_{A \in \mc P_n} \frac{1}{\mu(A)} \int_{A} \int_A |f(y) - f(x)| \, d\mu(y) \, d\mu(x) \\
&\leq \gamma^3  \sum_{A \in \mc P_n} \frac{1}{\lambda(A)} \int_{A} \int_A \|y - x\|_2 \, dy \, dx \\
&\leq \gamma^3  \sum_{A \in \mc P_n} \frac{1}{\lambda(A)} \int_{A} \int_A \|y - u^A\|_2 + \| u^A - x \|_2 \, dy \, dx, \\
\shortintertext{where $u^A$ is the upper corner of the set $A$: $u^A_i = \min \{ X_{j,i} : X_{j,i} \geq x_i \: \forall x \in A \}$ for $1 \leq i \leq d$ (and $u^A_i = 1$ if no such points exist).}
&= 2 \gamma^3 \sum_{A \in \mc P_n}  \int_A \|y - u^A\|_2  \, dy \\
&= 2 \gamma^3 \int_{[0,1]^d} \|y - u^{A_y}\|_2  \, dy,
\end{align*}
where $A_y$ denotes the $A \in \mc P_n$ containing $y$. 

Taking expectations and applying Cauchy-Schwarz, we get
\begin{align*}
\E \squa{\sup_{\| f \|_{\on{Lip}}\leq 1} \| \E[f \mid \mc F_n] - f \|_{L^1(\mu)}} &\leq 2 \gamma^3 \E\squa{\int_{[0,1]^d} \|y - u^{A_y}\|_2 \, dy} \\
&\leq 2 \gamma^3 \sqrt{\E\squa{\int_{[0,1]^d} \|y - u^{A_y}\|_2^2 \, dy}} \\
&= 2 \gamma^3\sqrt d \sqrt{\E\squa{\int_{[0,1]^d} (y_1 - u^{A_y}_1)^2 \, dy}}
\shortintertext{Since every $A \in \mc P_n$ is an axis-parallel box, $u^{A_y}_1 = \min \{ X_{j,1} : X_{j,1} \geq y_1 \}$; so the integral depends only on the 1st coordinate. Denoting the order statistics of the values $X_{1,1},\dots,X_{n,1}$ as $0 = Y_0 < Y_1 < \cdots < Y_n < Y_{n+1} = 1$, this is}
&= 2 \gamma^3\sqrt d \sqrt{\sum_{k=0}^n \E \squa{\int_{Y_k}^{Y_{k+1}} (y- Y_{k+1})^2 \, dy}} \\
&= 2 \gamma^3\sqrt d \sqrt{\sum_{k=0}^n \E \squa{\frac{(Y_{k+1}-Y_k)^3}{3}}}
\shortintertext{The distances between successive order statistics of uniform random variables on $[0,1]$ have Beta($1$,$n$) distribution. So this is}
&\ls \sqrt d \sqrt{(n+1) \frac{1}{(n+1)(n+2)(n+3)}} \\
&\leq \frac{\sqrt d}{n}.\qedhere
\end{align*}
\end{proof}

\end{document}